\documentclass[11pt]{article} 
\usepackage[utf8]{inputenc}
\usepackage[T1]{fontenc}
\usepackage[twoside,left=3.66cm,right=3.66cm,bottom=3cm,top=2.66cm]{geometry}
\usepackage{microtype}

\usepackage{amsmath,amsthm,amsfonts,amssymb}
\usepackage{mathrsfs}
\usepackage{cite}
\usepackage{enumerate}
\usepackage{cite}
\usepackage[all,2cell]{xy} \SilentMatrices
\usepackage{pdftricks}
\usepackage{epsfig,color}
\usepackage{mathtools}
\usepackage{sectsty}
\usepackage{anyfontsize}

\usepackage{framed}
\usepackage{color}
\definecolor{shadecolor}{gray}{0.875}
\definecolor{col}{RGB}{42, 95, 151}
\usepackage[colorlinks=true, linkcolor=col, citecolor=col, filecolor = col, menucolor = col, urlcolor = col]{hyperref}

\theoremstyle{plain}
\newtheorem{theorem}{Theorem}[section]
\newtheorem*{lemma*}{Lemma}
\newtheorem{lemma}[theorem]{Lemma}
\newtheorem*{theorem*}{Theorem}
\newtheorem{proposition}[theorem]{Proposition}
\newtheorem*{proposition*}{Proposition}
\newtheorem{corollary}[theorem]{Corollary}
\newtheorem*{corollary*}{Corollary}

\theoremstyle{definition}
\newtheorem{remark}[theorem]{Remark}
\newtheorem*{remark*}{Remark}
\newtheorem*{definition*}{Definition}

\newtheorem*{example*}{Example}

\newtheorem*{question*}{Question}

\def\RR{{\mathbb R}}

\def\c1{\operatorname{c_1}}
\def\c2{\operatorname{c_2}}

\def\CC{{\mathbb C}}
\def\ZZ{{\mathbb Z}}

\def\QQ{{\mathbb Q}}
\def\PP{{\mathbb P}}

\def\O{{\mathcal O}}

\def\Z{{\mathcal{Z}}}

\def\F{{\mathscr F}}

\def\+{\oplus}                   
\def\*{\otimes}

\def\Gr{{Gr}}

\def\eff{\operatorname{Eff}}
\def\nef{\operatorname{Nef}}

\def\Pic{\operatorname{Pic}}

\title{Nef cycles on some hyperk\"ahler fourfolds}
\author{John Christian Ottem}
\date{}

\begin{document}
\maketitle
\thispagestyle{empty}
\vspace{-1cm}

\begin{abstract}
 We study the cones of surfaces on varieties of lines on cubic fourfolds and Hilbert schemes of points on K3 surfaces. From this we obtain new examples of nef cycles which fail to be pseudoeffective. \end{abstract}
\thispagestyle{empty}

\section*{Introduction}
The cones of curves and effective divisors are essential tools in the study of algebraic varieties. Here the intersection pairing between curves and divisors allows one to interpret these cones geometrically in terms of their duals; the cones of nef divisors and `movable' curves respectively. In intermediate dimensions however, very little is known about the behaviour of the cones of effective cycles and their dual cones. Here a surprising feature is that nef cycles may fail to be pseudoeffective, showing that the usual geometric intuition for `positivity' does not extend so easily to higher codimension. In the paper \cite{delv}, Debarre--Ein--Lazarsfeld--Voisin presented examples of such cycles of codimension two on abelian fourfolds. 

In this paper, we show that similar examples can also be found on certain hyperk\"ahler manifolds (or `irreducible holomorphic symplectic varieties'). For these varieties, the cones of curves and divisors are already well-understood thanks to several recent advances in hyperk\"ahler geometry \cite{AV,B,BHT,BM,markman}. Our main result is the following:

\begin{theorem}\label{mainthm}
Let $X$ be the variety of lines of a very general cubic fourfold. Then the cone of pseudoeffective $2$-cycles on $X$ is strictly contained in the cone of nef 2-cycles. \end{theorem}

On such a hyperk\"ahler fourfold, the most interesting cohomology class is the second Chern class $c_2(X)$, which  represents a positive rational multiple of the Beauville--Bogomolov form on $H^2(X,\ZZ)$. This fact already implies that the class carries a certain amount of positivity, because it intersects products of effective divisors non-negatively (see \eqref{boucksom}). To prove Theorem \ref{mainthm} we show that this class is in the interior of the cone of nef cycles, but not in the cone of effective cycles. The main idea of the proof is to deform $X$ to the Hilbert square of an elliptic K3 surface. Here the claim for $c_2(X)$ follows from the fact that it has intersection number zero with the fibers of a Lagrangian fibration.

Thanks to B. Bakker, R. Lazarsfeld, B. Lehmann, U. Rie\ss, C. Vial, C. Voisin and L. Zhang for useful discussions. I am grateful to D. Huybrechts for his helpful comments and for inviting me to  Bonn where this paper was written.

\section{Preliminaries}

We work over the complex numbers. For a smooth projective variety $X$,  let $N_k(X)$ (resp. $N^{k}(X)$) denote the $\RR$-vector space of dimension (resp. codimension) $k$ cycles on $X$ modulo numerical equivalence. In $N_k(X)$ we define the pseudoeffective cone $\overline\eff_k(X)$ to be the closure of the cone spanned by classes of $k$-dimensional subvarieties. A class $\alpha\in N_k(X)$ is said to be {\em big} if it lies in the interior of $\overline \eff_k(X)$. This is equivalent to having $\alpha= \epsilon h^{\dim X-k}+e$ for $h$ the class of a very ample divisor; $\epsilon>0$; and $e$ an effective 2-cycle with $\RR$-coefficients. 

A codimension $k$-cycle is said to be {\em nef }if it has non-negative intersection number with any $k$-dimensional subvariety. We let $\nef^k(X)\subset N^k(X)$ denote the cone spanned by nef cycles; this is the dual cone of $\overline \eff_k(X)$. For the varieties in this paper, it is known that numerical and homological equivalence coincide, so we may consider these as cones in $H_{2k}(X,\RR)$ and $H^{2k}(X,\RR)$ respectively. If $Y\subset X$ is a subvariety, we let $[Y]\in H^*(X,\RR)$ denote its corresponding cohomology class.

For a variety $X$, we denote by $X^{[n]}$ the Hilbert scheme parameterizing length $n$ subschemes of $X$.

\subsection{Hyperk\"ahler fourfolds}

We will study effective 2-cycles on certain hyperk\"ahler varieties (or holomorphic symplectic varieties). By definition, such a variety is a smooth, simply connected algebraic variety admitting a non-degenerate holomorphic two-form $\omega$ generating $H^{2,0}(X)$. In dimension four there are currently two known examples of hyperk\"ahler manifolds up to deformation: Hilbert schemes of two points on a K3 surface and generalized Kummer varieties.

A hyperk\"ahler manifold $X$ carries an integral, primitive quadratic form  $q$ on the cohomology group $H^2(X,\ZZ)$ called the Beauville--Bogomolov form. The signature of this  form is $(3,b_2(X)-3)$, and $(1,\rho-1)$ when restricted to the Picard lattice $\Pic(X)=H^{1,1}(X)\cap H^2(X,\ZZ)$. For the hyperk\"ahler fourfolds considered in this paper it is known that the second Chern class $c_2(X)=c_2(T_X)$ represents a positive rational multiple of the Beauville--Bogomolov form \cite{SV}. It follows from this and standard properties of the Beauville--Bogomolov form that \begin{equation}\label{boucksom}c_2(X)\cdot D_1\cdot D_2\ge 0\end{equation} for any two distinct prime divisors $D_1,D_2$ on $X$.

\subsection{Specialization of effective cycles}

In the proof of Theorem \ref{mainthm}, we will need a certain semi-continuity result for effective cycles. This result is likely well-known to experts, but we include it here for the convenience of the reader and for future reference.

\def\X{\mathcal X}
\def\Y{\mathcal Y}

\begin{proposition}\label{deform}
Let $f:\X\to T$ be a smooth projective morphism over a smooth variety $T$ and suppose that $\alpha\in H^{k,k}(\X,\ZZ)$ is a class such that $\alpha|_{\X_t}$ is effective on the general fiber. Then $\alpha|_{\X_t}$ is effective for every $t\in T$.
\end{proposition}
\begin{proof}This follows basically from the theory of relative Hilbert schemes. The main point is that there are at most finitely many components $\rho_i:H_i\to T$ of the Hilbert scheme parameterizing cycles supported in the fibers of $f$ with cohomology class $\alpha$. For such a component $H_i$, let $\pi_i:\Z_i\to H_i$ denote the universal family of $H_i$. This fits into the diagram
$$
\xymatrix{
\Z_i \ar[r] \ar[d]^{\pi_i} &  \X \ar[d]^{f}\\
H_i \ar[r]^{\rho_i} & T
}
$$where $\pi_i$ is flat. Let $T'=\bigcup_i \rho_i(H_i)$, where the union is taken over all indices $i$ such that $\rho_i$ is not surjective. $T'$ is a proper closed subset of $T$. 

Let $t_0\in T$ be any point, and let $t\in T-T'$. By assumption the class $\alpha|_{\X_t}$ is represented by an effective cycle $Z_t$ on $\X_t$. By construction, there exists a component $H$ of the above Hilbert scheme, a universal family $\pi:\Z\to H$, a point $h\in H$ so that $Z_t=\Z_h$, and so that $\rho(H)=T$.  We have $\Z\subset \X'$, where $\X'=\X\times_T H$. Note that we have two sections, $[\Z]$ and $\rho^*\alpha$, of the local system $R^{2n-2k}\pi_*\ZZ$. These agree over $\rho^{-1}(t)$ so they agree everywhere. It follows that the  cycle $\Z|_{\X'_{h_0}}\in CH^{n-k}(\X_{t_0})$ is an effective cycle on $\X_{t_0}$ with class $\alpha|_{t_0}$, as desired.\end{proof}

By a limit argument, we get the following result which says that just like in the case of divisors, the effective cones can only become larger after specialization:

\begin{corollary}With the notation of Proposition \ref{deform}, if $\alpha|_t$ is big on a very general fiber, it is big on every fiber.\end{corollary}

%
%

\section{The variety of lines of a cubic fourfold}
Let $Y\subset \PP^5$ be a smooth cubic fourfold. The variety of lines $X=F(Y)$ on $Y$ is a smooth 4-dimensional subvariety of the Grassmannian $\Gr(2,6)$. A fundamental result due to Beauville and Donagi \cite{BD} says that $X$ is a holomorphic symplectic variety. Moreover, $X$ is deformation-equivalent to the Hilbert square $S^{[2]}$ of a K3 surface. 

From the embedding of $X$ in the Grassmannian we obtain natural cycle classes on $X$. In particular, if $U$ denotes the universal subbundle on $Gr(2,6)$, we can consider the Chern classes $g=c_1(U^\vee)|_X$ and $c=c_2(U^\vee)|_X$. Note that $g$ coincides with the polarization of $X$ given by the Pl\"ucker embedding of $X$ in $\PP^{14}$. When $Y$ is very general, a Noether--Lefschetz type argument shows that the vector space of degree four Hodge classes $H^{2,2}(X)\cap H^4(X,\QQ)$ is two-dimensional, generated by $g^2$ and $c$. We have the following intersection numbers:
\begin{equation}\label{gc}
g^4=108,\quad g^2\cdot c=45,\quad c^2=27. 
\end{equation}The cubic polynomial defining $Y$ shows that $X$ is the zero-set of a section of the vector bundle $S^3U^\vee$ on $Gr(2,6)$. Using this description, a standard Chern class computation shows that 
$$c_2(X)=5g^2-8c.$$
See for example \cite{amerik} or \cite{SV} for detailed proofs of these statements.
 
\subsection{Surfaces in the variety of lines}
There are several interesting surfaces on the fourfold $X$. The most basic of these are the restrictions of codimension two Schubert cycles from the ambient Grassmannian $Gr(2,6)$. In terms of $g^2$ and $c$, these cycles are given by $g^2-c$ and $c$. Moreover, since on $Gr(2,6)$ every effective cycle is nef (this is true on any homogeneous variety), also their classes remain nef and effective when restricted to $X$.

There are two natural surfaces on $X$ with class proportional to $g^2-c$. Fixing a general line $l\subset Y$, the surface parameterizing lines meeting $l$ represents the class $\frac13(g^2-c)$. Also, the variety of lines of `second type' (that is, lines with normal bundle $\O(1)^2\oplus \O(-1)$ in $Y$) is a surface with corresponding class $5(g^2-c)$ (see \cite{BD}).

The class $c$ is also represented by an irreducible surface. Indeed, for a general hyperplane $H=\PP^4\subset \PP^5$  the {\em Fano surface}$$\Sigma_H=\{[l]\in X \, | \,l\subset H\}$$ is a smooth surface of general type with $[\Sigma_H]=c$, corresponding to the lines in the cubic threefold $H\cap X$. 

We also have the following less obvious example. We can view $Y$ as the general hyperplane section of some cubic {fivefold} $V\subset \PP^6$. The variety of planes in $V$ is a smooth surface $F_2(V)$. As $Y$ is general, there is an embedding $F_2(V)\to F(Y)$ given by associating a plane to its intersection with the hyperplane section. The class of the image is a surface of general type with class $63c$ (see \cite{IM}). 

%

By the following result of Voisin \cite{voisin}, the class $c=c_2(U^\vee)$ lies the boundary in the cone of effective 2-cycles on $X$:

\begin{lemma}[Voisin]\label{voisin} Let $X$ be the variety of lines on a very general cubic fourfold. Then the class $c$ is extremal in the effective cone of 2-cycles.
\end{lemma}
This result deserves a few comments. First of all, we note that even though the class $c$ is the restriction of an extremal Schubert cycle on $Gr(2,6)$, there is a priori no reason to expect that it should remain extremal when restricted to $X$. To see how subtle this is, we mention that Voisin \cite[Theorem 2.9]{voisin} showed that in the case where $X$ is instead the variety of lines of a cubic {\em fivefold} $Y$, the corresponding class $c$ of lines in a hyperplane section of $Y$ {is} in fact big on $X$ (in the sense that it lies in the interior of the effective cone).

It is in fact quite surprising that the class $c$ is extremal as the surface $\Sigma_H\subset X$ behaves in many ways like a complete intersection. For instance, varying the hyperplane $H$, it is clear that it deforms in a large family covering $X$, and $\Sigma_H\cdot V>0$ for any other surface $V\subset X$. In fact, Voisin showed that the surface above gives a counterexample to a question of Peternell \cite{Pet08}, who asked whether a smooth subvariety with ample normal bundle is big. This question has a positive answer for curves and divisors\cite{ottJEMS}. In the current setting, the normal bundle of $\Sigma_H$ coincides with the restriction of $U^\vee$, which is ample on $\Sigma_H$.

\begin{remark}Even though the subvariety $\Sigma_H$ is very positively embedded in $X$, it is not an `ample subscheme' in the sense of \cite{ottem}. Essentially, \cite{ottem} defines a smooth subvariety $Y\subset X$ to be ample if large powers of the ideal sheaf kill cohomology of every coherent sheaf $\F$  in degrees $<\dim Y$ (that is,
$H^k(X,I_Y^m\otimes \F)=0$ for $k<\dim Y$ and $m \gg 0$). See \cite{ottem} for basic results on such subschemes. To see why $\Sigma_H$ is not ample, we can use the fact that ample subschemes satisfy Lefschetz hyperplane theorems on rational homology\cite{ottem}, so that their Betti numbers agree $b_i(X)=b_i(Y)$ for $i<\dim Y$. However, in our case $X$ is simply connected, whereas both $\Sigma_H$ and $F_2(V)$ have non-zero first Betti number. As ample subschemes also have ample normal bundles, this raises the question whether Peternell's question has a positive answer when restricted to ample subschemes, that is, whether smooth ample subvarieties have big cycle classes. As usual, the answer is affirmative for curves and divisors.
\end{remark}

\subsection{Lagrangian submanifolds}

Voisin's proof of Lemma \ref{voisin} is quite interesting: it uses the fact that $\Sigma_H$ is a Lagrangian submanifold of $X$. Indeed, the class of the 2-form $\omega$ is the image of a primitive cohomology class $\sigma^{3,1}\in  H^4(X,\QQ)_{prim}$ under the incidence correspondence $P\subset F\times X$; from this it follows that $\omega$ restricts to 0 on any surface with class proportional to $c$ (see \cite[Lemma 1.1]{voisin}). 

Now, the $(2,2)$-form $\omega\wedge \overline \omega$ either vanishes on a surface $V$ (in which case $V$ is Lagrangian), or restricts to a multiple of the volume form on $V$. In the latter case, the integral $\int_V \omega\wedge \overline \omega$ is strictly positive. Hence any surface $V$ so that $\omega|_Y$ vanishes cannot be homologous to a cycle of the form $\epsilon h^2+e$ where $h$ is ample and $e$ is effective. Consequently, since numerical and homological equivalence coincide on $X$, we see that the class $\omega\wedge \overline \omega$ defines a supporting hyperplane of the effective cone of surfaces.

This argument works in greater generality (using the Hodge--Riemann relations) and one obtains the following:
\begin{proposition}\label{Lagrangian}
Let $X$ be a smooth projective variety of dimension $n$ and let $\omega$ be a closed $p$-form representing a non-zero class $[\omega]\in H^{p,0}(X)$. If $V\subset X$ is a $p$-dimensional subvariety so that $\omega|_V=0$, then $V$ is not {homologous} to a big cycle.

In particular, if $X$ is an holomorphic symplectic variety of dimension four, and $V\subset X$ is a Lagrangian surface, then $[V]$ is in the boundary of $\overline{\eff}_2(X)$.
\end{proposition}

%

This result can be used to prove extremality of other cycle classes. Two standard examples of Lagrangian subvarieties are: (i) Any surface $V\subset X$ with $H^{2,0}(V)=0$ (e.g., a rational surface); and (ii)  $C^{[2]}\subset S^{[2]}$ for a curve $C$ in a K3 surface $S$. For instance, if $Y\subset \PP^5$ is a cubic fourfold containing a plane $P$, then the dual plane $P^\vee\subset X=F(Y)$ parameterizing lines in $P$ is extremal. (See also \cite[\S 3.2]{rempel}). 

\subsection{Proof of Theorem \ref{mainthm}}

Voisin's result gives one face of $\overline\eff_2(X)$. To bound the other half of the effective cone of $X$, we will consider the class $c_2(X)=5g^2-8c$. Note that this class is already quite positive, since it intersects products of divisors $D_1D_2$ non-negatively (cf. equation \eqref{boucksom}). In fact, this class will be shown to lie in the interior of the cone of nef 2-cycles, and is {\em strictly nef}, in the sense that $c_2(X)\cdot Z>0$ for every irreducible surface $Z\subset X$. We will first show that it is not big. We first consider a special case:

\begin{proposition}\label{cnnotbig}
Let $X$ be a hyperk\"ahler manifold of dimension $2n$, admitting a Lagrangian fibration. Then $c_k(X)$ is not big, for $1\le k\le n$. \end{proposition}

\begin{proof}
Let $f:X\to B$ denote this fibration. By a theorem of Matsushita, $\dim B=n$ and the general fiber $A$ of $f$ of is an $n$-dimensional abelian variety. In particular, $\Omega^1_A\simeq\O_A^n$. Let $Y=A\cap H_1\cap \cdots \cap H_{n-k}$, where $H_i\in |h|$, and $h$ is a very ample line bundle on $X$. The restriction of the normal bundle sequence
$$
0\to T_A|_Y=\O_Y^n\to T_{X}|_Y \to N_A|_Y= \O_Y^n\to 0 
$$to $Y$ shows that $c_k(X)\cdot A\cdot h^{n-k}=0$. Now if $c_k$ is a big cycle, then it is numerically equivalent to $\epsilon h^k+e$, for some $\epsilon>0$ and $e$ an effective cycle with $\RR-$coefficients. However, the cycle $A\cdot h^{n-k}$ is nef (and non-zero), and so it has strictly positive intersection number with $h^k$. Hence $c_k(X)$ cannot be big. 
\end{proof}
\begin{lemma}\label{hilbc2}
Let $X=S^{[n]}$ for a very general K3 surface $S$. Then the Chern classes $c_k(X)$ are not big for $1\le k\le n$.
\end{lemma}
\begin{proof}
By Proposition \ref{deform}, it suffices to prove this for a special K3 surface $S$. Choosing an elliptic K3 surface $S\to \PP^1$, we obtain a Lagrangian fibration $S^{[n]}\to (\PP^1)^{[n]}=\PP^n$, and so the claim follows from Proposition \ref{cnnotbig}.
\end{proof}

\begin{corollary}\label{c2notbig}
Let $X=F(Y)$ for a very general cubic fourfold $Y$. Then the second Chern class $c_2(X)=5g^2-8c$ is not big.\end{corollary}

\begin{proof}
As in the previous lemma, it suffices to exhibit one cubic fourfold for which the statement is true. We specialize to a Pfaffian cubic fourfold $Y$ so that $F(Y)=S^{[2]}$ for a degree 14 K3 surface $S$ (provided that $Y$ does not contain any planes \cite{BD}). Choosing $Y$ so that $S$ contains an elliptic fibration, we obtain the required abelian surface fibration $S^{[2]}\to (\PP^1)^{[2]}=\PP^2$ on $F(Y)$, and we conclude using Proposition \ref{cnnotbig}. \end{proof}

Combining this  with Voisin's result, we obtain the following bound for the effective cone $\overline{\eff}_2(X)$:
\begin{corollary}
Let $X$ be the variety of lines on a very general cubic fourfold. Then \begin{equation}\label{eff}\RR_{\ge 0} c+\RR_{\ge 0}(g^2-c)\subseteq {\eff}_2(X)\subseteq
\RR_{\ge 0} c+\RR_{\ge 0}(g^2-\tfrac85c)
\end{equation}
\end{corollary}

\begin{proof}[Proof of Theorem \ref{mainthm}]
By the previous corollary, we have that $\overline{\eff}_2(X)\subseteq \RR_{\ge0}c+\RR_{\ge 0}c_2(X)$. Using the intersection numbers \eqref{gc}, we find by duality:
\begin{eqnarray*}
\nef^2(X)&\supseteq& \left( \RR_{\ge0}c+\RR_{\ge 0}(5g^2-8c)\right)^\vee\\
&=& \RR_{\ge0}(20c-g^2)+\RR_{\ge 0}(3g^2-5c)\\
&\supsetneq &\overline{\eff}_2(X) 
\end{eqnarray*}\end{proof}
In the above proof, the class $c_2(X)$ plays an important role. The theorem implies that $c_2(X)$ is in the interior of the nef cone, but not big. We do not know if $c_2(X)$ is not pseudoeffective (i.e., a limit of effective cycles) or if there are special cubic fourfolds for which $c_2(X)$ is effective. 

The role of $c_2(X)$ is also interesting from the viewpoint of generic {\em non-projective} deformations $X$ of the fourfold above, since in that case we have $H^{2,2}(X,\CC)\cap H^4(X,\QQ)=\QQ c_2(X)$. In this case, Verbitsky \cite{verbitsky} proved that $X$ contains no analytic subvarieties of positive dimension at all.

\begin{remark}
In his thesis \cite{rempel}, Max Rempel also considered  examples of pseudoeffective cycles on hyperk\"ahler fourfolds, in particular also on the variety of lines of a cubic fourfold. Among his many results, Rempel shows that the class $3g^2-5c$ (which is proportional to $c_2(X)-\frac15g^2$) has no effective multiple. The proof involves a nice geometric argument using Voisin's rational self-map $\varphi:X\dashrightarrow X$. \end{remark}

\begin{remark}[An alternative definition of nefness]
In the Nakai--Moishezon criterion, a line bundle $L$ is ample if and only if for every $r=1,\ldots,\dim X$ we have $L^r\cdot Z>0$ for every subvariety of dimension $r$. This suggests the following naive fix for the definition of nefness for any cycle class: $\gamma\in N^k(X)$ could be defined to be nef if the restriction cycle $i^*\gamma\in N^k(Y)$ is pseudoeffective for every subvariety $i:Y\hookrightarrow X$. In particular, taking $i$ to be the identity map, $\gamma$ would itself be pseudoeffective on $X$. This is related to Fulger and Lehmann's notion of {\em universal pseudoeffectivity} \cite{FL} where the $f^*\gamma$ is required to be pseudoeffective for any morphism $f:Y\to X$. It is an interesting question is whether these notions are equivalent. 
\end{remark}

\section{Generalized Kummer varieties}
Let $X$ denote a generalized Kummer variety of dimension 4 of an abelian surface $A$. Recall that these are defined as the fiber over 0 of the addition map $\Sigma:A^{[3]}\to A$. In this case, the group $H^2(X,\ZZ)$ can be identified with $H^2(A,\ZZ)\oplus \ZZ e$, where the class $e$ is one half of the divisor corresponding to non-reduced subschemes in $A^{[3]}$. As before, $c_2(X)$ represents a positive multiple of the Beauville--Bogomolov form \cite{HT}. 

There are 81 distinguished Kummer surfaces $Z_\tau$ on $X$, whose classes are linearly independent in $H^4(X,\QQ)$: $Z_0$ is the closure of the locus of points $(0, a,-a)$ with $a\in A-0$, and the other $Z_\tau$ are the translates of $Z_0$ via the 3-torsion points $A[3]$. Hassett--Tschinkel \cite[Proposition 5.1]{HT} showed that 
$$
c_2(X)=\frac1{3}\sum_{\tau\in A[3]}[Z_\tau].
$$In particular, $3c_2(X)$ represents an effective cycle. Moreover, using the same argument as in Lemma \ref{hilbc2} we obtain

\begin{proposition}
Let $X$ be a generic projective generalized Kummer fourfold. Then the second Chern class $c_2(X)$ is effective, but not big.
\end{proposition}

{

\bibliographystyle{plain}

}

%
%

{Department of Mathematics, University of Oslo, Box 1053, Blindern, 0316 Oslo, Norway}

{{\it Email:} \verb"johnco@math.uio.no"}

\end{document}